\documentclass{amsart}
\usepackage{amsmath}
\usepackage{amssymb}
\usepackage{color}

\DeclareMathOperator{\Id}{Id}

\newtheorem{theorem}{Theorem}[section]
\newtheorem{lemma}[theorem]{Lemma}
\newtheorem{proposition}[theorem]{Proposition}
\newtheorem{corollary}[theorem]{Corollary}
\theoremstyle{definition}

\newtheorem{example}[theorem]{Example}

\theoremstyle{remark}
\newtheorem{remark}[theorem]{Remark}
\numberwithin{equation}{section}

\def\separa{\hbox to 14 truecm{\hrulefill}}

\author[P. Danchev]{Peter Danchev}
\address{Institute of Mathematics and Informatics, Bulgarian Academy of Sciences, 1113 Sofia, Bulgaria}
\email{danchev@math.bas.bg}
\thanks{The first author was partially supported by the Bulgarian National Science Fund under Grant KP-06 No. 32/1 of December 07, 2019.}

\author[E. Garc\'\i a]{Esther Garc\'\i a}
\address{ Departamento de Matem\'{a}tica  Aplicada, Ciencia e Ingenier\'{\i}a de los Materiales y Tecnolog\'{\i}a Electr\'onica,
Universidad Rey Juan Carlos, 28933 M\'{o}s\-to\-les (Madrid), Spain}
\email{esther.garcia@urjc.es}

\author[M. G\'omez Lozano]{Miguel G\'omez Lozano}
\address{Departamento de \'Algebra, Geometr\'{\i}a y
Topolog\'{\i}a, Universidad de M\'alaga, 29071 M\'alaga, Spain}
\thanks{The second two authors were partially supported by  MTM2017-84194-P (AEI/FEDER, UE),  and by the Junta de Andaluc\'{\i}a FQM264.}
\email{miggl@uma.es}

\begin{document}

\title[Decompositions of Matrices into ...]{Decompositions of Matrices into \\ Potent and Square-Zero Matrices}
\maketitle

\begin{abstract} In order to find a suitable expression of an arbitrary square matrix over an arbitrary finite commutative ring, we prove that every such a matrix is always representable as a sum of a potent matrix and a nilpotent matrix of order at most two when the Jacobson radical of the ring has zero-square. This somewhat extends results of ours in Lin. \& Multilin. Algebra (2021) established for matrices considered on arbitrary fields. Our main theorem also improves on recent results due to Abyzov et al. in Mat. Zametki (2017), \v{S}ter in Lin. Algebra \& Appl. (2018) and Shitov in Indag. Math. (2019).

\end{abstract}



\section{Introduction and Fundamentals}

We start the frontier of this paper by recalling that an element $x$ of an arbitrary ring $R$ is said to be {\it nilpotent} if there is an integer $i>0$ such that $x^i=0$ whereas an element $y$ from $R$ is said to be {\it potent}, or more exactly {\it $m$-potent}, if there is a natural number $m\geq 2$ with $y^m=y$. In particular, all the idempotents are always $2$-potent elements.

Our current work is devoted to the further study, firstly somewhat initiated in \cite{BCDM}, of decomposing square matrices as a sum of a potent and a nilpotent. Concretely, a brief retrospection of the most important results in this direction is as follows:

In \cite{BCDM} was proven that each matrix from the ring $\mathbb{M}_n(\mathbb{F}_2)$ of $n\times n$ matrices over the field $\mathbb{F}_2$ of two elements is a sum of an idempotent matrix and a nilpotent matrix -- even something more, if the matrix ring $\mathbb{M}_n(F)$ over an arbitrary field $F$ possesses this property, then $F\cong \mathbb{F}_2$. This result was substantially strengthened by \v{S}ter in \cite{S} who proved that $\mathbb{M}_n(\mathbb{F}_2)$ is actually a sum of an idempotent matrix and a nilpotent matrix of index at most $4$. Lately, this result was significantly improved by Shitov in \cite{Sh} for certain matrix sizes $n$. Moreover, an important work was done by de Seguins Pazzis in \cite{dSP}, where a valuable discussion on the decomposition of a matrix as a sum of an idempotent and a square-zero matrix is provided.

On the other vein, Abyzov and Mukhametgaliev showed in \cite{AM} that, for all naturals $n\geq 1$, any element of the ring $\mathbb{M}_n(F)$ is presented as a sum of a nilpotent and a $q$-potent element, provided that $F$ is a field of cardinality $q$ -- specifically, in \cite[Theorem 2]{AM} was showed that some square matrix over finite fields are expressible as a sum of a potent and a nilpotent but the order of the existing nilpotent is, in general, greater than $2$. Also, a recent paper \cite{B} by Breaz deals with the more exact presentation of matrices over fields of odd cardinality $q$ as a sum of a $q$-potent matrix and a nilpotent matrix of order $3$. Besides, it was constructed in \cite[Example 6]{B} an ingenious example of a $3\times 3$ matrix over the field $\mathbb{F}_3$ of three elements that cannot be presented as the sum of a $3$-potent and a nilpotent matrix of order $2$ (in other terms, the latter matrix is also called {\it square-zero} or, equivalently, {\it zero-square}). Furthermore, improving the aforementioned results from \cite{B}, we establish in \cite{DGL} that each square matrix over any infinite field as well as each matrix over some special finite fields can be expressed as a sum of a potent matrix and a square-zero matrix.

So, a question which logically arises is of whether or not our results in \cite{DGL} could be expanded for some kinds of (finite) commutative rings, that is, is every square matrix over a finite commutative ring of square-prime characteristic decomposed as a sum of a potent matrix and a square-zero matrix (for example, for rings of the sort $\mathbb{Z}_{p^2}$ for some fixed prime $p$)? To keep a record straight, we notice that a similar representation of such a matrix ring over $\mathbb{Z}_4$ already exists in terms of a nilpotent of order less than or equal to $8$ and an idempotent (see, e.g., \cite{S}). Even more generally, it was established in \cite[Lemma 1]{AM} and \cite[Theorem 4]{AM} that, for all $n,m\in \mathbb{N}$, the matrices in $\mathbb{M}_n(\mathbb{Z}_{p^m})$ are presentable as the sum of a nilpotent matrix and a $p$-potent matrix, whenever $p$ is a prime. However, the exact bound (of course, if it eventually exists) of the existing nilpotent matrix is not explicitly calculated yet.

And so, being seriously motivated by the present idea, in what follows, we shall completely resolve \cite[Problem 2]{DGL} even in a more general setting (see, e.g., Theorem~\ref{main}) and, besides, we shall strengthen the previously mentioned achievements from \cite{AM}, \cite{S} and \cite{Sh}, respectively.

Likewise, for completeness of the introductory section, we refer to the bibliography of the cited by us articles and also concretize that some related results can be found by the interested reader in \cite{BM} and \cite{JS} along with the given references therewith, respectively.

\section{Main Results and Conjecture}

We begin here with the following simple but useful claim.

\begin{lemma}\label{invertible} Let $R$ be a finite unital commutative ring.
For every invertible matrix $A\in \mathbb{M}_n(R)$ there exists $m\in\mathbb{N}$ such that $A^{m-1}=\Id$ and $A^m=A$.
\end{lemma}

\begin{proof} Let $A$ be an invertible matrix in  $\mathbb{M}_n(R)$ and  consider the set of matrices $\{A^0,A^1,\dots, A^n,\dots\}$. Since this set is finite, there exists $k<l$ such that $A^k=A^l$, and since $A$ is invertible $\Id=A^{l-k}$. The claim now follows by taking $m-1=l-k$.
\end{proof}

The following result generalizes \cite[Corollary 3.2]{DGL}, where it was shown that every matrix over a finite field is a sum of potent matrix and a zero-square matrix by using a different approach. The result of this paper is entirely based on the primary rational canonical form of a matrix (\cite[VII.Corollary 4.7(ii)]{H}), which states that every matrix $A\in \mathbb{M}_n(\mathbb{F})$ where $\mathbb{F}$ is a field is similar to a direct sum of companion matrices of prime power polynomials $p_1^{m_{11}}, \dots, p_s^{m_{sk_s}}\in \mathbb{F}[x]$ where each $p_i$ is prime (irreducible) in $\mathbb{F}[x]$. The matrix $A$ is uniquely determined except for the order of the companion matrices of the $p_i^{m_{ij}}$ along its main diagonal. The polynomials $p_1^{m_{11}}, \dots, p_s^{m_{sk_s}}$ are called {\it the elementary divisors} of the matrix $A$.

\begin{proposition}\label{field} Let $\mathbb{F}$ be a finite field. For any matrix $A\in \mathbb{M}_n(\mathbb{F})$ there exists $k\in \mathbb{N}$ such that $A=P+N$, where $N^2=0$, $P^k=P$, $E=P^{k-1}$ is an idempotent with $PE=EP=P$ and $EN=NE=N$.
\end{proposition}

\begin{proof}
Let us consider the  primary rational canonical form of the matrix $A$. Also, let us split our argument between elementary divisors $q_i(x)$ of $A$ with $q_i(0)\ne 0$ and those with $q_i(0)=0$:

\begin{itemize}
  \item [(i)] Any elementary divisor  $q_i(x)$ with $q_i(0)\ne 0$ gives rise to an invertible companion matrix $C_i$. By Lemma~\ref{invertible} there exists $k_i\in \mathbb{N}$ such that $C_i^{k_i}=C_i$ and $C_i^{k_i-1}=\Id$. Let us denote $P_i:=C_i$ and define $N_i$ as the zero matrix.
  \item [(ii)] Let us suppose that $q_i(x)$ is an elementary divisor (power of an irreducible polynomial in $\mathbb{F}[x]$) such that $q_i(0)=0$. This implies that $q_i(x)=x^{i_s}$ for certain $i_s\in \mathbb{N}$ and its associated companion matrix is of the form
      $$
      C_i=\left(
          \begin{array}{cccc}
            0 & 0 & \dots  & 0 \\
            1 & 0 &  &  \vdots\\
             & \ddots & \ddots &  \\
            0 &  & 1 & 0\\
          \end{array}
        \right)\in M_{i_s}(\mathbb{F}),
              $$
      i.e., it is a nilpotent Jordan block.
      \begin{itemize}
        \item [(ii.1)]
      If $i_s\ge 2$,  write
      $$
      P_i=\left(
          \begin{array}{cccc}
            0 & 0 & \dots  & 1 \\
            1 & 0 &  &  \vdots\\
             & \ddots & \ddots &  \\
            0 &  & 1 & 0\\
          \end{array}
        \right) \quad \text{and}\quad  N_i:=\left(
          \begin{array}{cccc}
            0 & 0 & \dots  & -1 \\
            0 & 0 &  &  \vdots\\
             & \ddots & \ddots &  \\
            0 &  & 0 & 0\\
          \end{array}
        \right)$$ Notice that $P_i$ is an invertible matrix  and by \ref{invertible}  there exists $k_i\in \mathbb{N}$ such that $P_i^{k_i}=P_i$ and $P_i^{k-1}=\Id$ with $C_i=P_i+N_i$.
       \item[(ii.2)] If $i_s=1$, then
        $$
        C_i=\left(
              \begin{array}{c}
                0 \\
              \end{array}
            \right).
        $$
 \end{itemize}
    \end{itemize}

Let $Q$ be the invertible matrix in $\mathbb{M}_n(\mathbb{F})$ such that $Q^{-1}AQ$ is decomposed into its primary rational canonical form (suppose without loss of generality that the blocks corresponding to (ii.2) are written together as the last zero block):
{\tiny
\begin{align*}
  Q^{-1}AQ&=\left(
    \begin{array}{c|c|c|c|c}
      C_1 & 0 &\cdots & 0  &0 \\
      \hline
       0 & C_2 &\cdots & 0  &0 \\
       \hline
       \vdots &\vdots& \ddots  & \vdots  &   \\
      \hline
      0 & 0&\cdots   &C_r & 0  \\
      \hline
     0 & 0&   &   & 0 \\
    \end{array}
  \right)=\underbrace{\left(
    \begin{array}{c|c|c|c|c}
      P_1 & 0 &\cdots & 0  &0 \\
      \hline
       0 & P_2 &\cdots & 0  &0 \\
       \hline
       \vdots &\vdots& \ddots  & \vdots  &   \\
      \hline
      0 & 0&\cdots   &P_r & 0  \\
      \hline
     0 & 0&   &   & 0 \\
    \end{array}
  \right)}_{P'}+\underbrace{\left(
    \begin{array}{c|c|c|c|c}
      N_1 & 0 &\cdots & 0  &0 \\
      \hline
       0 & N_2 &\cdots & 0  &0 \\
       \hline
       \vdots &\vdots& \ddots  & \vdots  &   \\
      \hline
      0 & 0&\cdots   &N_r & 0  \\
      \hline
     0 & 0&   &   & 0 \\
    \end{array}
  \right)}_{N'}
\end{align*}
}

Since each $P_i$ satisfies $P_i^{k_i-1}=\Id$, we have that $(P')^k=P'$ for $k=1+\prod_{i=1}^{r}(k_i-1)$ and therefore $E':=P'^{k-1}$ is an idempotent of  $\mathbb{M}_n(\mathbb{F})$ ($E'^2=(P'^{k-1})^2=P'^kP'^{k-2}=P'^{k-1}=E'$). By construction, $N'^2=0$ and, since $P_i^k=\Id$, in each block it must be that $E'N'=N'E'=N'$.

Finally, for $P:=QP'Q^{-1}$, $N:=QN'Q^{-1}$ and $E:=QE'Q^{-1}$ we have that $E$ is an idempotent of $M_n(\mathbb{F})$ satisfying the properties $E=P^{k-1}$, $A=P+N$, $P^k=P$, $N^2=0$, $E N=N E=N $ and $E P=P E=P $. In particular, $P=EPE$ is invertible in the subring $E \mathbb{M}_n(\mathbb{F})E$ (note that $E \mathbb{M}_n(\mathbb{F})E$ is a unital ring with unit $E$). This, in turn, implies that $P$ is strongly regular, as required.
\end{proof}

The next property of lifting idempotents is well-known, but we list the statement here only for the sake of completeness and for the convenience of the readers.

\begin{lemma}\cite[27.1]{AF}\label{liftingidempotents}
Let $R$ be a ring and let $I$ be a nilpotent ideal of $R$. Then any idempotent of $R/I$ lifts to an idempotent of $R$.
\end{lemma}

The following two technicalities on lifting special elements are the key for the establishment of our further results.

\begin{lemma}\label{liftnil} Let $R$ be a ring and let $I$ be a nilpotent ideal of $R$. Let us suppose that $\overline a\in R/I$ has zero-square and that $\bar a$ is a von Neumann regular element in $R/I$. Then the element $\overline  a$ lifts to an element of $R$ with zero-square.
\end{lemma}

\begin{proof}
Since $\overline a$ is a von Neumann regular element of zero square, there exists $\overline b\in R/I$ such that $\bar a \bar b\bar a=\bar a$, $\bar b\bar a\bar b=\bar b$ and $\bar b^2=0$ (see \cite[Lemma 2.4]{GLAS}). Let us consider the idempotent $\bar e=\bar a\bar b\in R/I$. Notice that $\bar e\bar a (1-\bar e)=\bar a$, because $\bar a\bar e=\bar 0$.
By consulting with \cite[27.1]{AF}, the element $\bar e$ lifts to an idempotent $e\in R$. If now we take any representant $a$ of $\bar a$ in $R$, we will have that $ea(1-e)\in R$ has zero-square and $\overline{ea(1-e)}=\bar e\bar a (1-\bar e)=\bar a$, as claimed.
\end{proof}

For completeness of the exposition, let us recall now that a unital commutative ring is said to be {\it a local ring} if it contains a unique maximal ideal, say $M$. In that case the factor-ring $R/M$ is a field, called {\it the residue field} of $R$ -- cf. \cite[Definition 1.2.9]{BF}. Moreover, any finite commutative ring with identity $R$ can be expressed as a direct sum of local rings and the decomposition is unique up to a permutation of the direct summands (see, e.g., \cite[Theorem 3.1.4]{BF}).

\begin{lemma}\label{liftpotent} Let $R$ be a unital commutative local ring such that its unique maximal ideal $M$ has $M^2=0$. Let $S= \mathbb{M}_n(R)$. Take $P,E\in S$ such that $E$ is an idempotent of $S$ and $P=EPE$ and suppose that there exists $k\in \mathbb{N}$ such that $\overline P^{k-1}=\overline E \in S/rad(S)$. Then there exists a prime $p>1$ such that $P^{(k-1)p}=E$. In particular, $P^{(k-1)p+1}=P$ and $P$ is invertible in $ESE$.
\end{lemma}

\begin{proof}
Since $\overline P^{k-1}=\overline E$, there exists $U\in rad(S)=\mathbb{M}_n(M)$ such that $E=P^{k-1}+U$. Multiplying on the left and on the right by $E$, we can suppose that $U=EUE$. We know that $R/M$ is a finite field of certain prime characteristic $p$. Thus $p\cdot 1\in M$, so we have that $pM\subset M^2=0$. We, consequently, calculate that $$P^{(k-1)p}=(E-U)^p=E+\sum_{i=1}^p(-1)^i\binom{p}{i}U^i=E,$$
as expected.
\end{proof}

So, we arrive at our central result on decomposing any matrix over special finite commutative rings into a potent matrix and a zero-square matrix.

\begin{theorem}\label{main}
Let $R$ be a finite commutative ring such that its Jacobson radical has zero-square. Then every matrix $A$ in $\mathbb{M}_n(R)$ can be expressed as $P+N$, where $P$ is a potent matrix and $N$ is a nilpotent matrix with $N^2=0$.
\end{theorem}

\begin{proof} We know with the aid of the comments alluded to above that $R$ is a direct sum $R=\bigoplus R_i$, where each $R_i$ is a local ring. Then one finds that the decomposition $\mathbb{M}_n(R)=\bigoplus_i \mathbb{M}_n(R_i)$ holds, so that we can express $A$ as a direct sum of matrices over local rings.

Suppose without loss of generality that $R$ is a local ring. Let us denote $S:=\mathbb{M}_n(R)$ and let us decompose $A\in \mathbb{M}_n(R)$ into a potent and a zero-square matrix. Let $I$ be the unique maximal ideal of $R$. By hypothesis, one calculates that $I^2=0$ because $I$ coincides with the Jacobson radical $rad(R)$ of $R$. Clearly, $J:=rad(\mathbb{M}_n(R))=\mathbb{M}_n(rad(R))=\mathbb{M}_n(I)$ and hence $J^2=0$.

Let us consider the residue class of $A$ modulo $J$: In fact, $\bar A\in  S/J\cong  \mathbb{M}_n(R/I)$. Since $R/I$ is a finite field, by virtue of Proposition~\ref{field} there exists $k\in \mathbb{N}$ such that $\bar A=\hat P+\hat N$ with $\hat P^{k}=\hat P$, $\hat N^2=0$, and $\hat E:=\hat P^{k-1}$ is an idempotent with $\hat E\ \hat N=\hat N\ \hat E=\hat N $ and $\hat E\ \hat P=\hat P\ \hat E=\hat P$.

Now, with Lemma~\ref{liftingidempotents} at hand, there exists an idempotent $E$ of $S$ such that $\overline E=\hat E$. Let us consider $P\in S$ such that $\overline P=\hat P$. Since $\hat E\Hat P=\hat P\hat E=\hat P$, we have that $\overline{EP}=\overline{PE}=\overline P$ and we can suppose (by replacing $P$ by $EPE$) that $EP=PE=P$. Applying
Lemma~\ref{liftnil} there exists $N\in S$ such that $N^2=0$ with $\overline N=\hat N$ and we can suppose, again replacing $N$ by $ENE$, that $N=EN=NE$. Furthermore, employing Lemma~\ref{liftpotent}, there exists a prime $p>1$ such that $P^{(k-1)p}=E$.

Let us take $V\in J$ such that $A=P+N+V$, and write
 $$
 A=\underbrace{P+EVE+(1-E)VE+EV(1-E)}+\underbrace{N+(1-E)V(1-E)}.\qquad{(*)}
 $$

\noindent (1) Let us show that $P+EVE+(1-E)VE+EV(1-E)$ is a potent element of $S$: In fact, notice that for any $n\in \mathbb{N}$
\begin{align*}
    ((P&+EVE)+(1-E)VE+EV(1-E))^n\\
   &=(P+EVE)^n+(1-E)VE(P+EVE)^{n-1}+(P+EVE)^{n-1}EV(1-E)
\end{align*}
because $EVE$, $(1-E)VE$, $EV(1-E)$ and $(1-E)V(1-E)$ belong to an ideal of zero-square. Moreover, since $P$ is invertible in $ESE$, the matrix $P+EVE$ is invertible and, therefore, it is $k$-potent. Furthermore, $P+EVE$ satisfies the conditions of Lemma~\ref{liftpotent}, so we detect that $((P+EVE)^{(k-1)p})=E$ and $(P+EVE)^{(k-1)p+1}=P+EVE$. Then
\begin{align*}
((P+EVE)&+(1-E)VE+EV(1-E))^{(k-1)p+1}\\&=(P+EVE)^{(k-1)p+1}+(1-E)VE(P+EVE)^{(k-1)p}\\&+(P+EVE)^{(k-1)p}EV(1-E)\\&=P+EVE+(1-E)VE+EV(1-E).
\end{align*}

\noindent (2) Let us show that $N+(1-E)V(1-E)$ has zero-square: Indeed, since $EN=NE=N$, one has that $N(1-E)V(1-E)=0=(1-E)V(1-E)N$, and since $(1-E)V(1-E)$ belongs to an ideal of zero-square, one has that
$$(N+(1-E)V(1-E))^2=N^2+((1-E)V(1-E))^2=0.$$

Therefore, equality $(*)$ provides the desired decomposition of $A$ into a potent matrix and a zero-square matrix, as wanted.
\end{proof}

Since $\mathbb{Z}_{p^2}$ is a unital commutative local ring whose Jacobson radical has zero-square, we immediately obtain the following consequence, which completely resolves \cite[Problem 2]{DGL} when $p=2$, i.e., for the ring $\mathbb{Z}_4$.

\begin{corollary}\label{Zp2} For all natural numbers $n$ and primes $p$, every matrix in $\mathbb{M}_n(\mathbb{Z}_{p^2})$ can be expressed as $P+N$, where $P$ is a potent matrix and $N$ is a matrix with $N^2=0$.
\end{corollary}

The next construction sheds a further light on the more concrete decomposition of such a type.

\begin{example}
Let us check in an example how this decomposition successfully works. Suppose $S=\mathbb{M}_8(\mathbb{Z}_4)$ and consider the following matrix
  $$
 A= \underbrace{\left(
    \begin{array}{ccc|ccc|c|c}
      0 &  &  1&  &  &  & & \\
      1 &0  & 0 &  & 0 &  & 0 & 0\\
      0 & 1 & 1 &  &  &  & & \\
      \hline
       &  &  & 0 &  & 0 &  &\\
       & 0  &  & 1 & 0 & 0 &  0&0\\
       &  &  &  0& 1 & 0 &  &\\
       \hline
       &  0 &  &  & 0 &  & 0  &0\\
 \hline
     & 0 &  &   & 0 &   & 0 & 0\\
    \end{array}
  \right)}_{\bar A}+2B,
  $$
where $B$ is any matrix in $S$,
  $$
  B=\left(
    \begin{array}{ccc|ccc|c |c}
       &  &   &  &  &  & & \\
        &  a &  &  & b &  &   c&  d\\
        &   &   &  &  &  & & \\
      \hline
       &  &  &   &  &   &  &\\
       & e  &  &   &  f &   &  g& h\\
       &    &  &   &   &   &   & \\
       \hline
       &  i  &  &  &   j&  &   k &l \\
 \hline
       &  m &  &   &   n&   &  p &q  \\
    \end{array}
  \right),
  $$
where $a,b,\dots, q$ denote matrices of the appropriate sizes.

If we regard $A$ as a matrix over $\mathbb{Z}_2$, we will obtain $\bar A$, whose elementary divisors are $x^3+x^2+1$, $x^3$, $x$, $x$. As in the proof of Theorem~\ref{main}, we add and subtract the element $1$ in the adequate position of the second diagonal box in order to transform the companion matrix of $x^3$ into an invertible matrix plus a zero-square matrix:
  $$
  A=\underbrace{\left(
    \begin{array}{ccc|ccc|c|c}
      0 &  &  1&  &  &  & & \\
      1 & 0 & 0 &  & 0 &  & 0 & 0\\
      0 & 1 & 1 &  &  &  & & \\
      \hline
       &  &  & 0 &  & \boxed{1} &  &\\
       & 0  &  & 1 & 0 & 0 &  0&0\\
       &  &  &  0& 1 & 0 &  &\\
       \hline
       &  0 &  &  & 0 &  & 0  &0\\
 \hline
     & 0 &  &   & 0 &   & 0 & 0\\
    \end{array}
  \right)}_{\hat P=\bar A+e_{4,7}}+\underbrace{\left(
    \begin{array}{ccc|ccc|c|c}
        &  &  &  &  &  & & \\
        &  0&  &  & 0 &  & 0 & 0\\
        &  &  &  &  &  & & \\
      \hline
       &  &  &   &  & \boxed{3} &  &\\
       & 0  &  &  & 0 &   &  0&0\\
       &  &  &   &  &   &  &\\
       \hline
       &  0 &  &  & 0 &  & 0  &0\\
 \hline
     & 0 &  &   & 0 &   & 0 & 0\\
    \end{array}
  \right)}_{\hat N=3e_{4,7}}+2B
  $$
  The matrix $\hat P=\bar A+e_{4,7}$ satisfies
  $$\hat P^{42}= \left(\begin{array}{ccc|ccc|c|c}
        1&  &  &  &  &  & & \\
        &   1&  &  & 0 &  &   &  \\
        &  &  1&  &  &  & & \\
      \hline
       &  &  &   1&  &   &  &\\
       &    &  &  & 1 &   &   & \\
       &  &  &   &  &   1&  &\\
       \hline
       &  0 &  &  & 0 &  & 0  &0\\
 \hline
     & 0 &  &   & 0 &   & 0 & 0\\
    \end{array}
  \right)
  $$
(it is worthwhile noticing that the first diagonal box to the $7^{th}$-power is the identity when regarded as a matrix over $\mathbb{Z}_2$, the second diagonal box to the $3^{rd}$-power is the identity when regarded as a matrix over $\mathbb{Z}_2$, so globally we need $7\times 3\times 2$ to get a common identity over $\mathbb{Z}_4$).

Following the proof of the theorem, let us denote $E=\hat P^{42}$, which is clearly an idempotent of $S$.

Now
  $$
 A=\underbrace{\hat P+E(2B)E+(1-E)(2B)E+E(2B)(1-E)}_P+\underbrace{\hat N+(1-E)(2B)(1-E)}_N
 $$
where
\begin{align*}
 \bullet\ & P=\hat P+E(2B)E+(1-E)(2B)E+E(2B)(1-E)=\\
       &=\left(
    \begin{array}{ccc|ccc|c|c}
      0 &  &  1&  &  &  & & \\
      1 &  & 0 &  & 0 &  & 0 & 0\\
      0 & 1 & 1 &  &  &  & & \\
      \hline
       &  &  & 0 &  & \boxed{1} &  &\\
       & 0  &  & 1 &  & 0 &  0&0\\
       &  &  &  0& 1 & 0 &  &\\
       \hline
       &  0 &  &  & 0 &  & 0  &0\\
 \hline
     & 0 &  &   & 0 &   & 0 & 0\\
    \end{array}
  \right)+2\left(
    \begin{array}{ccc|ccc|c |c}
       &  &   &  &  &  & & \\
        &  a &  &  & b &  &   c&  d\\
        &   &   &  &  &  & & \\
      \hline
       &  &  &   &  &   &  &\\
       & e  &  &   &  f &   &  g& h\\
       &    &  &   &   &   &   & \\
       \hline
       &  i  &  &  &   j&  &   0 &0 \\
 \hline
       &  m &  &   &   n&   & 0 &0  \\
    \end{array}
  \right)\\
  \bullet \ & N=\hat N+(1-E)(2B)(1-E)=\\
  &=\left(
    \begin{array}{ccc|ccc |c |c}
       &  &   &  &  &  & & \\
        &  0  &  &  & 0  &  &   0&0  \\
        &   &   &  &  &  & & \\
      \hline
       &  &  &   &  & \boxed{3}  &  &\\
       &  0 &  &   & 0  &   &  0& 0\\
       &    &  &   &   &   &   & \\
       \hline
       &  0  &  &  &  0 &  &     0&  0\\
 \hline
       &  0 &  &   &  0 &   &   0&  0 \\
    \end{array}
  \right)   +2\left(
    \begin{array}{ccc|ccc|c |c}
       &  &   &  &  &  & & \\
        &  0  &  &  & 0  &  &   0& 0 \\
        &   &   &  &  &  & & \\
      \hline
       &  &  &   &  &    &  &\\
       &  0 &  &  &0    &   & 0 & 0\\
       &    &  &   &   &   &   & \\
       \hline
       &   0 &  & & 0   &  &    k & l \\
 \hline
       &  0 &  &  &0    &   & p &q  \\
    \end{array}
  \right)
\end{align*}
and $P^{43}=P$ and $N^2=0$.

In this example, $\bar A$ was already in its rational canonical form over $\mathbb{Z}_2$. For otherwise, one should consider the appropriate invertible matrix $Q$ such that $Q^{-1}\bar AQ$ is a direct sum of blocks corresponding to elementary divisors and adapt the idempotent matrix $\hat E$ and the nilpotent matrix $\hat N$.
\end{example}

The following construction unambiguously illustrates that the square-prime characteristic of the ring is an essential condition and cannot be dropped off.

\begin{remark}\label{exhibit}
There are matrices over $\mathbb{Z}_{2^3}$ that do not admit a decomposition into potent + zero-square. For example, the matrix

$$
A=2\Id\in \mathbb{M}_n(\mathbb{Z}_{2^3})
$$

\medskip

\noindent does not admit such a decomposition. Otherwise, since $A^2\ne 0$ there would exist a non-zero potent matrix $P$ and a zero-square matrix $N$ such that $A=P+N$. Then $P^4=((A-N)^2)^2=(4\Id-4N)^2=0$, which is not possible if $P$ is potent and non-zero, thus establishing our claim.

On the other side, Theorem~\ref{main} remains no longer true for finite commutative rings of characteristic $p^2$ for some arbitrary but fixed prime $p$. In fact, it suffices to find a finite commutative ring $R$ of characteristic $p^2$ having an element $a$ with $a^3=0$ and $a^2\ne 0$. For example, consider the ring $R=\mathbb{Z}_4[x]/I$ where $I$ is the ideal generated by the polynomial $(x^2+x+1)^3$. The characteristic of $R$ is then exactly $4$. Choose $a=(x^2+x+1)+I\in R$, and let us consider similarly to above the matrix $A=a\Id\in \mathbb{M}_n(R)$ for some $n\in \mathbb{N}$. This matrix $A$ has the properties $A^2\ne 0$ and $A^3=0$, whence with the help of the same argument as above it surely cannot be decomposed into the sum of a potent and a zero-square nilpotent. This concludes our arguments.
\end{remark}

In order to generalize Theorem~\ref{main} to commutative rings of the form $\mathbb{Z}_{p^r}$ for some natural number $r\geq 2$, we first are going to show that potent elements lift modulo a nilpotent ideal. Our proof follows the ideas of the classical lifting of idempotents (see, for instance, \cite[Proposition 27.1]{AF}).

\begin{proposition}\label{potentslift}
Let $R$ be a  finite ring and let $I$ be a nilpotent ideal of $R$ of index $n$. Let us suppose $A\in R$ is such that $\overline A\in R/I$ is a potent element of $R/I$. Then there exists $B\in R$ such that $\overline A=\overline B$ and $B$ is potent in $R$.
\end{proposition}

\begin{proof}Suppose that $\overline A^t=\overline A\in R/I$. Then $A^t-A\in I$ and, therefore,
\begin{align*}
0&=(A^t-A)^n=\sum_{k=0}^n (-1)^{n-k}\binom{n}{k}A^{kt}A^{(n-k)}\\&= (-1)^nA^n-\sum_{k=1}^{n} (-1)^{n-k+1}\binom{n}{k}A^{n+(t-1)k}\\&=
 (-1)^nA^n-A^{n+(t-1)}\left(\sum_{k=1}^n (-1)^{n-k+1}\binom{n}{k}A^{(k-1)(t-1)}\right)
\end{align*}
Let $T:=\sum_{k=1}^n (-1)^{2n-k+1}\binom{n}{k}A^{(k-1)(t-1)}$. Therefore, $A^n=A^{n+(t-1)}T$. Define $E:=A^{n(t-1)}T^n$. Notice that $EA=AE$. Let us show that $E$ is an idempotent of $R$. In fact, one sees that

\begin{align*}
E&=A^{n(t-1)}T^n=A^{n(t-1)+(t-1)}T^{n+1}=A^{n(t-1)+2(t-1)}T^{n+2}=\cdots=\\&=A^{n(t-1)+n(t-1)}T^{n+n}=E^2
\end{align*}

From the above calculations, we may write that
$$
E=(EA)(EA^{n(t-1)-1}T^n),
$$
and so we get that
$EA$ is an invertible element in $ERE$ (note that $ERE$ is a unital ring with identity element $E$). Decompose $A=EA+(1-E)A$. Since $EA^n=A^n$, it must be that $((1-E)A)^n=0$. On the other hand, if we choose $k$ such that $t^k\ge n$ and taking into account that $\bar A$ is $t$-potent, it follows that
$$
\bar A=\bar A^t=\dots=\bar A^{t^k}=(\overline{EA}+\overline{(1-E)A})^{t^k}=\overline{EA}^{t^k}+\overline{((1-E)A}^{t^k}=\overline{EA}^{t^k}\in \overline{ER},
$$
so $\bar A=\overline{EA}+\overline{(1-E)A}\in \overline {ER}$ implies $\bar A=\overline{EA}$.

To conclude that $EA$ is potent, it suffices to consider the finite set $$\{EA, (EA)^2,\dots, (EA)^r,\dots \}$$ to get that $(EA)^l=(EA)^m$ for some $l<m\in \mathbb{N}$, and from the invertibility of $EA$ we get that $(EA)^{m-l}=E$, so $(EA)^{m-l+1}=EA$, as asserted.
\end{proof}

So, we are ready to proceed by proving with the promised generalization.

\begin{corollary} Let $n,r$ be two natural numbers. Then every matrix $A$ in $\mathbb{M}_n(\mathbb{Z}_{p^r})$ can be expressed as $P+N$, where $P$ is a potent matrix and $N$ is a matrix such that $N^2\in \mathbb{M}_n(p^2\mathbb{Z}_{p^r})$.
\end{corollary}

\begin{proof} Let $R=\mathbb{M}_n(\mathbb{Z}_{p^r})$ and let us consider the nilpotent ideal $I$ of $R$ generated by $p^2\Id$, i.e., $I=\mathbb{M}_n(p^2\mathbb{Z}_{p^r})$. Then $R/I\cong \mathbb{M}_n(\mathbb{Z}_{p^2})$ and thus it satisfies the hypothesis of Theorem~\ref{main}. Consequently, there exists $\overline P, \overline N\in R/I$ such that $\bar A=\bar P+\bar N$, where $\overline P$ is potent and $\overline N^2=\overline 0$. By making use of Proposition~\ref{potentslift}, we can lift $\bar P$ to a potent matrix $P$ of $R$. Take any matrix $N\in R$ such that $N$ modulo $I$ coincides with $\bar N$ (in particular we have that $N^2\in I$). Finally, there exists $V\in I$ such that $A=P+N+V$, where
\begin{itemize}
  \item $P$ is potent
  
  \medskip
  
  \item $(N+V)^2=N^2+NV+VN+V^2\in I=\mathbb{M}_n(p^2\mathbb{Z}_{p^r}).$
\end{itemize}
\end{proof}

It is worth noticing that this result extends Corollary~\ref{Zp2} in a more general vision (notice that when $r=2$, the ideal $\mathbb{M}_n(p^2\mathbb{Z}_{p^2})$ is zero and so $N$ has zero-square).

\medskip

We finish off our work with the following conjecture which addresses Remark~\ref{exhibit} quoted above.

\vskip0.5pc

\noindent{\bf Conjecture.} Suppose $m,n\geq 2$ are natural numbers and $p$ is a prime. Then every matrix in $\mathbb{M}_n(\mathbb{Z}_{p^m})$ is a sum of a potent and of a nilpotent of order at most $m$.

\medskip

Note that our results stated above (especially Corollary~\ref{Zp2}) completely settled the problem for $m=2$. In this aspect, can we refine our machinery and results for finite commutative rings of characteristic $p^m$?

\vskip3.0pc

\bibliographystyle{plain}

\end{document}